\documentclass{article}
\usepackage{amssymb,amsmath,amsthm,mathtools}
\usepackage{fullpage}
\usepackage{graphicx}
\usepackage{tikz,pgfplots}
\usepackage{hyperref}

\newtheorem{theorem}{Theorem}

\newtheorem{lemma}[theorem]{Lemma}
\newtheorem{cor}[theorem]{Corollary}
\newtheorem{obs}[theorem]{Observation}

\newtheoremstyle{case}{}{}{}{}{}{:}{ }{}
\theoremstyle{case}

\def\R{\mathbb{R}}

\newcommand{\conv}{\mathsf{conv}}

\def\ie{i.e.\ }
\def\etal{{\it et~al.}\,}

\begin{document}

\title{New Lower Bounds for Tverberg Partitions with Tolerance in the Plane}

\author{Sergey Bereg\thanks{Department of Computer Science,
        University of Texas at Dallas. This research is supported in part by NSF award CCF-1718994. }
        \and
        Mohammadreza Haghpanah$^*$}
\maketitle

\begin{abstract}
Let $P$ be a set $n$ points in a $d$-dimensional space. 
Tverberg's theorem says that, if $n$ is at least $(k-1)(d+1)+1$, 
then $P$ can be partitioned into $k$ sets whose convex hulls intersect. 
Partitions with this property are called {\em Tverberg partitions}. 
A partition has tolerance $t$ if the partition remains a Tverberg partition after removal of any set of $t$ points from $P$.
Tolerant Tverberg partitions exist in any dimension provided that $n$ is sufficiently large.
Let $N(d,k,t)$ be the smallest value of $n$ such that tolerant Tverberg partitions exist for any set of $n$ points in $\R^d$. 
Only few exact values of $N(d,k,t)$ are known. 

In this paper we establish a new tight bound for $N(2,2,2)$. 
We also prove many new lower bounds on $N(2,k,t)$ for $k\ge 2$ and $t\ge 1$. 
\end{abstract}

\section{Introduction}

The classical Tverberg's theorem~\cite{t-grt-66} states that a
sufficiently large set of points in $\R^d$ 
can be partitioned into subsets such that their convex hulls intersect. 
For the history and recent advances around Tverberg's theorem, we refer the reader to 
\cite{tverberg50,barany2018tverberg,de2019discrete,matouvsek2002lectures}.

\begin{theorem}[Tverberg~\cite{t-grt-66}]
Any set $P$ of at least $(k-1)(d+1)+1$ points in $\R^d$ can be partitioned into 
$k$ sets $P_1,P_2,\dots,P_k$ whose convex hulls intersect, i.e.~\begin{equation}
\label{eqtverberg}
\bigcap _{i=1}^k \conv(P_i) \ne \emptyset. 
\end{equation}
\end{theorem}

In 1972, David Larman proved the first tolerant version of Tverberg's theorem: 
any set of $2d+3$ points in $\R^d$ can be partitioned into two sets $A$ and $B$ such that 
their convex hulls intersect with tolerance one, i.e.,~for 
any point $x\in \R^d$, $\conv(A\setminus\{x\})\cap \conv(B\setminus\{x\})\ne \emptyset$. 
Garc{\'{\i}}a{-}Col{\'{\i}}n~\cite{garcia07} generalized it for any tolerance.
Sober{\'{o}}n and Strausz~\cite{soberon12} generalized it further
for partitions into many sets.
The general problem can be stated as follows. 

{\bf Problem} (Tverberg partitions with tolerance). 
Given positive integers $d,k,t$, find the smallest positive integer $N(d,k,t)$ such that 
any set $P$ of $N(d,k,t)$ points in $\R^d$ can be 
partitioned into $k$ subsets $P_1, P_2,\dots, P_k$ such that for any set $Y\subset P$ of at most $t$ points
\begin{equation} \label{eqttverberg}
\bigcap _{i=1}^k \conv(P_i\setminus Y) \ne \emptyset.
\end{equation}

We call a partition of $P$ into $k$ subsets $P_1, P_2,\dots, P_k$ {\em $t$-tolerant} 
if condition (\ref{eqttverberg}) holds for any set $Y$ of at most $t$ points. 
Several upper bounds for $N(d,k,t)$ are
known~\cite{garcia17,larman72,ms-attp-14,soberon12}. 
Some lower bounds for $N(d,k,t)$ are
known~\cite{garcia15,alfonsin01,soberon15}. 

\begin{theorem}[Ram{\'{\i}}rez-Alfons{\'{\i}}n~\cite{alfonsin01}]
For any $d\ge 4$
\begin{equation}
  N(d,2,1) \ge \left\lceil \frac{5d}{3} \right\rceil + 3. 
\end{equation}
\end{theorem}

\begin{theorem}[Garc{\'i}a-Col{\'i}n and Larman~\cite{garcia15}]
\begin{equation} \label{gl15}
N(d,2,t) \ge 2d+t+1. 
\end{equation}
\end{theorem}

\begin{theorem}[Sober{\'{o}}n~\cite{soberon15}]
\begin{equation} \label{s15}
N(d,k,t) \ge k(t+\lfloor d/2\rfloor +1). 
\end{equation}
\end{theorem}

In contrast to Tverberg's theorem, most bounds are not known to be tight.
The only known tight bounds are the following. 
Larman~\cite{larman72} proved $N(d,2,1)=2d+3$ for $d\le 3$. 
Forge, Las Vergnas and Schuchert~\cite{fls-10-01} proved $N(4,2,1)=11$.
For all $k\ge 2, t\ge 1$ and dimension one, 
$N(1,k,t)= k(t+2)-1$
by Mulzer and Stein~\cite{ms-attp-14}. 
In this paper we establish a new tight bound for $N(2,2,2)$. 

\begin{theorem} \label{t222}
$N(2,2,2) = 10$. 
\end{theorem}

We also improve the bound (\ref{s15}) for the plane.

\begin{theorem} \label{N2ktc}
Let $c$ be a positive integer.
For any integers $k\ge 2c$ and $t\ge c$ 
\begin{equation} \label{E2ktc}
N(2,k,t)\ge k(t+2)+c. 
\end{equation}
\end{theorem}

{\bf Remark}. The bound of Theorem \ref{N2ktc} can be stated without parameter $c$ 
\begin{equation} \label{E2ktc1}
N(2,k,t)\ge k(t+2)+\min(t,\lfloor k/2\rfloor). 
\end{equation}

If $d=2$ and $k=2$ then Equation (\ref{E2ktc}) provides a lower bound $N(2,2,t)\ge 2(t+2)+1$. 
We further improve the bound for $k=2$ and $t\ge 3$.

\begin{theorem} \label{N2kt2}
For any $t \ge 3$, $N(2,2,t)\ge 2t+6$. 
\end{theorem}

Several upper bounds for $N(d,k,t)$ are
known~\cite{garcia07,garcia17,soberon2018robust}. 
For example, Sober{\'o}n~\cite{soberon2018robust} proved
$N(d,k,t)=kt+O(\sqrt{t})$ for fixed $k$ and $d$.
Therefore, it is interesting to find lower and upper bounds for $P(d,k,t) = N(d,k,t)-kt$.  
Theorems \ref{t222},\ref{N2ktc}, and \ref{N2kt2} provide new 
lower bounds for $P(d,k,t)$ for $d=2$.

{\bf Remark}. Recently, we improved some lower bounds using computer pograms~\cite{bh-arpt-20}. 
For example, $N(2,2,t)\ge 2t+7$ for $5\le t\le 10$.

The rest of the paper is organized as follows. Section \ref{pre} introduces some observations and lemmas that will be used in our proofs. 
Theorems \ref{N2ktc} and \ref{N2kt2} are proven in Section \ref{thm2}.  
Theorem \ref{t222} is proven in Section \ref{thm222}.
Finally, we discuss the results and open problems in Section \ref{diss}.

\section{Preliminaries} \label{pre}

We start with a simple observation. 
\begin{obs} \label{sizet}
Every part in a $t$-tolerant partition has size at least $t+1$.
\end{obs}

The following lemma is simple, yet very useful in proving that a partition is not $t$-tolerant.  

\begin{lemma} \label{separ}
Let $P_1,P_2,\dots,P_k$ be a partition of a finite set $P$ in $\R^d$ and let $X$ be a subset of a set $P_i$ such that 
\begin{equation} \label{separ1}
\conv(X)\cap\conv(P\setminus X)=\emptyset
\end{equation}
and $|X|=m, |P_i|=t+m$. Then the partition of $P$ is not $t$-tolerant.
\end{lemma}

\begin{proof}
Let $Y=P_i\setminus X$. Then $|Y|=t$. Therefore 
\begin{equation}
\bigcap _{j=1}^k \conv(P_j\setminus Y) = \emptyset
\end{equation}
by Equation \ref{separ1} since one of the sets in the intersection is $X$ and every other set is a subset of $P\setminus X$.
\end{proof}

Two special cases of Lemma \ref{separ}.

\begin{lemma} \label{l1}
A partition $P_1,P_2,\dots,P_k$ of a finite set $P$ in $\R^d$ is not $t$-tolerant if one of the following conditions holds.
\begin{enumerate}
\item[{\rm (1)}]
A set $P_j$ contains a vertex of $\conv(P)$ and $|P_j|=t+1$.
\item[{\rm (2)}]
A set $P_j$ contains two points $p_i,p_i+1$ such that $p_i,p_i+1$ is an edge of $\conv(P)$ and $|P_j|=t+2$.
\end{enumerate}
\end{lemma}

\begin{proof}
The first claim follows from Lemma \ref{separ} by setting $X=\{p_i\}$ where $p_i\in P_j$ is a vertex of $\conv(P)$.
The second claim follows from Lemma \ref{separ} by setting $X=\{p_i,p_{i+1}\}$.
\end{proof}

\section{Proofs of Theorems \ref{N2ktc} and \ref{N2kt2}}
\label{thm2}

First, we prove Theorem \ref{N2ktc}.

\begin{proof} [Proof of Theorem \ref{N2ktc}]
Construct a set $P$ of $k(t+2)+c-1$ points in the plane as follows. 
Let $Q$ be a regular $(k(t+2)-1)$-gon with center at the origin $O$. 
Place $k(t+2)-1$ points $q_1,\dots,q_{k(t+2)-1}$ at the vertices of $Q$ and $c$ points $p_1,\dots,p_c$ close to the origin.
More specifically, we assume that $|p_1|,\dots,|p_c| <d$ where $d$ is the distance from the origin to the line $p_1p_k$.
Let $V=\{q_1,\dots,q_{k(t+2)-1}\}$. 
We assume that the vertices of $Q$ are sorted in clockwise order, see Figure \ref{f2ktc}. 

\begin{figure}[htb]
\centering
\includegraphics{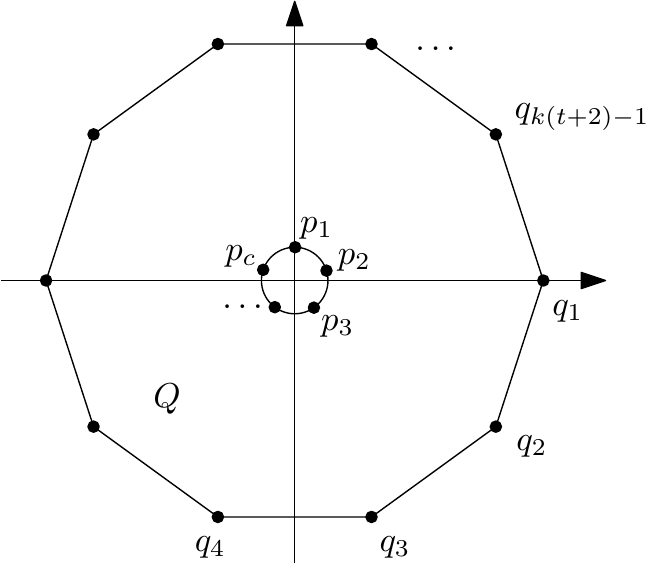}
\caption{A point set for Theorem \ref{N2ktc}.}
\label{f2ktc}
\end{figure}

We show that any partition of $P$ into $P_1,P_2,\dots,P_k$ is not $t$-tolerant. 
By Observation \ref{sizet}, we assume $|P_i| \geq t+1$ for all $i$. 
If $|P_i|=t+1$ for some $i$, then $P_i$ contains a point $q_j$ since $t+1>c$.
The partition is not $t$-tolerant by Lemma \ref{l1}(1). 

It remains to consider the case where all sets in the partition have size at least $t+2$.
At most $c-1$ sets in the partition may have size $\ge t+3$.
Since $k\ge 2c$ there exists a set $P_i$ of size $t+2$ such that $P_i\subset V$.
Since $|V|=k(t+2)-1$, there exist two points $q_a$ and $q_{a+j}$  with $j\le k-1$ such that the vertices between 
$q_a$ and $q_{a+j}$ are in $Q\setminus P_i$ (we assume here that $q_x=q_y$ if $x\equiv y\pmod{k(t+2)-1}$).
There is a set $P_m,m\ne i$ such that none of these vertices are in $P_m$.
Therefore the partition is not $t$-tolerant since $\conv\{q_a,q_b\}\cap \conv(P_m)=\emptyset$. 
Note that the distance from the origin to any point in $\{p_1,\dots,p_c\}$ is smaller than the distance from the origin to the line $q_aq_b$.
\end{proof}

Next we prove Theorem \ref{N2kt2}. 

\begin{figure}[htb]
\centering
\includegraphics{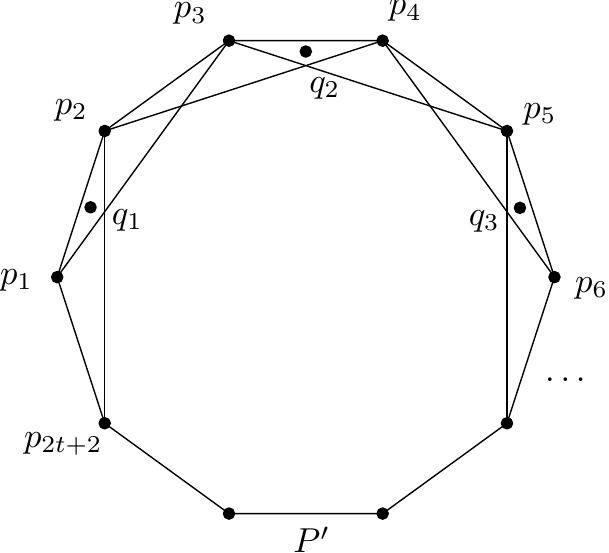}
\caption{A point set for Theorem \ref{N2kt2}.}
\label{f2kt2}
\end{figure}

\begin{proof} [Proof of Theorem \ref{N2kt2}]
Construct a set $P$ of $2t+5$ points in the plane as follows. 
First, place $2t+2$ points in convex position $P' = \{ p_1, p_2, \dots, p_{2t+2} \}$ where 
points are in clockwise order. Then place three points $Q=\{q_1, q_2, q_3\}$ inside 
the convex hull of $P'$ such that point $q_i,i=1,2,3$ is close to edge $p_{2i - 1}p_{2i}$ 
(it is required that $q_i$ is inside triangles $\bigtriangleup p_{2i-2}p_{2i-1}p_{2i}$ and $\bigtriangleup p_{2i-1}p_{2i}p_{2i+1}$),
see Figure \ref{f2kt2}.
We show that any partition of $P$ into $P_1,P_2$ is not $t$-tolerant. 

Without loss of generality, we assume $|P_1| \leq |P_2|$. 
Therefore, $|P_1| \leq t + 2$. 
By Observation \ref{sizet}, we assume $|P_1| \geq t+1$. 
Suppose if $|P_1| = t + 1$.
Since $t \geq 3$, $P_1$ contains a point of $P'$.
The partition is not $t$-tolerant by Lemma \ref{l1}(1). 

It remains to consider the case where $|P_1| = t + 2$. 
If set $P_1$ does not contain a point $q_i$, then $P_1\subset P'$ and $|P_1|\ge |P'|/2$.
Then $P_1$ contains two consecutive points of $P'$. 
The partition is not $t$-tolerant by Lemma \ref{l1}(2). 
Therefore $P_1\cap Q\ne \emptyset$.

Suppose $q_i \in P_1$ for some $i$. 
If $p_{2i-1} \in P_1$ or $p_{2i} \in P_1$ then the partition is not $t$-tolerant
using $X=\{ q_i, p_{2i-1} \}$ or $X=\{ q_i, p_{2i} \}$, respectively, and Lemma \ref{separ}.
Therefore, we assume that both points $p_{2i-1}$ and $p_{2i}$ are in $P_2$.

Let $m=|P_1 \cap \{ q_1, q_2, q_3 \}|$. Then $m\in\{1,2,3\}$. 
By the previous argument, set $P_2$ contains at least $m$ pairs of consecutive points of $P'$ 
(the pair $p_{2i-1},p_{2i}$ for each point $q_i\in P_1$). 
Note that $|P'\cap P_1|=t+2-m$ and $|P'\cap P_2|=t+m$. 
We will show that set $P_1$ contains a pair of consecutive points of $P'$.
Then, the partition is not $t$-tolerant by Lemma \ref{l1}(2). 

If $m=1$, then $|P'\cap P_1|=|P'\cap P_2|=t+1$ and set $P_2$ contains a pair of consecutive points of $P'$.
Then set $P_1$ contains a pair of consecutive points of $P'$.

Suppose $m=2$, say $q_i,q_j\in P_1$ where $i<j$. 
Then set $P_2$ contains $p_{2i-1}, p_{2i},  p_{2j-1}$, and $p_{2j}$.
There are two intervals $p_{2i+1},\dots,p_{2j-2}$ and $p_{2j+1},\dots,p_{2i-2}$ 
in $P'$ containing points of $P_1$. 
Note that each interval contains an even number of points.
If set $P_1$ does not contain a pair of consecutive points of $P'$, 
then set $P_1$ contains at most half of points in each interval. 
Then $|P'\cap P_1| <t$. This contradiction implies that 
set $P_1$ contains a pair of consecutive points of $P'$.

Finally, if $m=3$, then there exists only one interval $p_7,p_8,\dots,p_{2t+2}$ and set $P_1$ must contain a pair of consecutive points of $P'$; otherwise $|P'\cap P_1|<t+1$. 
The theorem follows.
\end{proof}

\section{Proof of Theorem \ref{t222}}
\label{thm222}

We found a special configuration of nine points to prove a lower bound for Theorem \ref{t222}.  
For this, we checked point configurations of all order types for $n=9$. 
Order types, introduced by Goodman and Pollack~\cite{gp-ms-83}, are useful for characterizing the
combinatorial properties of point configurations. 
An order type of a set of points $p_1,p_2,\dots,p_n$ in general position in the plane 
is defined using a mapping that assigns each triple of integers $i,j,k$ with 
$1\le i<j<k\le n$ the orientation (either clockwise or counter-clockwise) of the triple $p_i,p_j,p_k$. 
Two point sets $P_1$ and $P_2$ have the same {\em order type} if there is 
a bijection $\pi$ from $P_1$ to $P_2$ preserving this map, i.e., 
for any three distinct points $a,b,c$ in $P_1$, the orientation of $a,b,c$ 
and the orientation of $\pi(a),\pi(b),\pi(c)$ are the same. 
Aichholzer \etal~\cite{aak-eot-02} established that there are 158,817 order types for $n=9$.
We developed a program for testing each of these point sets whether it admits a 2-tolerant partition into two sets. 
The proof of the tolerance of a configuration found by the program turns out 
to be not simple even using the tools from Section \ref{pre}.

\begin{figure}[htb]
\centering
\includegraphics{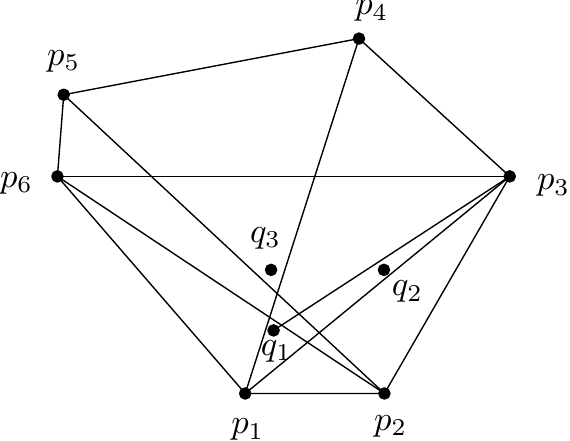}
\caption{A set $S$ of nine points for Lemma \ref{l222}.}
\label{nine}
\end{figure}

\begin{lemma} \label{l222}
There exists a set of nine points in the plane that does not admit a 2-tolerant partition into two sets. 
\end{lemma}

\begin{figure}[htb]
\centering
\includegraphics[]{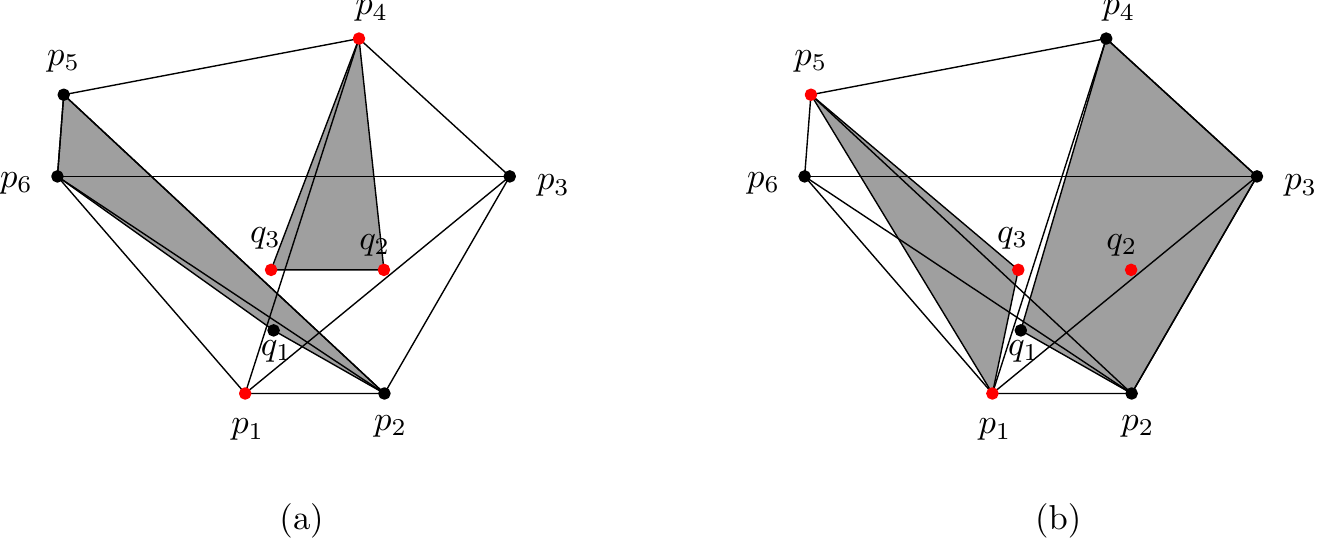}
\caption{Case 2a. The points of $A$ are red and the points of $B$ are black. 
The shaded polygons are the convex hulls of $A-C$ and $B-C$.}
\label{case2a}
\end{figure}

\begin{proof}
Consider a set $S$ of nine points shown in Fig.\ref{nine} 
where $S=P\cup Q, P=\{p_1,\dots,p_6\}$ and $Q=\{q_1,q_2,q_3\}$.
We will prove that any partition of $S=A\cup B$ is not 2-tolerant, \ie  
there exists a set $C=\{a,b\}\subseteq S$ such that 
\[ \conv(A\setminus C)\cap \conv(B\setminus C)=\emptyset. \]

We can assume that $|A|<|B|$. 
By Observation \ref{sizet}, we assume $|A| \geq 3$. 
Suppose that $|A|=3$. 
If $A=Q$, then take $C=\{p_1,p_2\}$.
If $A\ne Q$, then the partition of $S$ is not 2-tolerant by Lemma \ref{l1}(1).

It remains to analyze partitions with $|A|=4$. 
By Lemma \ref{l1}(2), we assume that $A$ does not contain two consecutive points $p_i,p_{i+1}$ (assuming $p_7=p_1$).
For example, this implies that $A$ contains at least one point of $Q$.
Consider the following cases depending on $|A\cap P|$.

{\bf Case 1}. Set $A$ contains exactly one point of $P$, say $p_i$. 
Then $Q\subset A$.
One of the sets $\{p_1,p_2,p_3\}$, $\{p_3,p_4,p_5\}$, $\{p_4,p_5,p_6\}$, say set $X$, 
does not contain $p_i$. The partition is not 2-tolerant by Lemma \ref{separ} using $X$.

\begin{figure}[htb]
\centering{\includegraphics[]{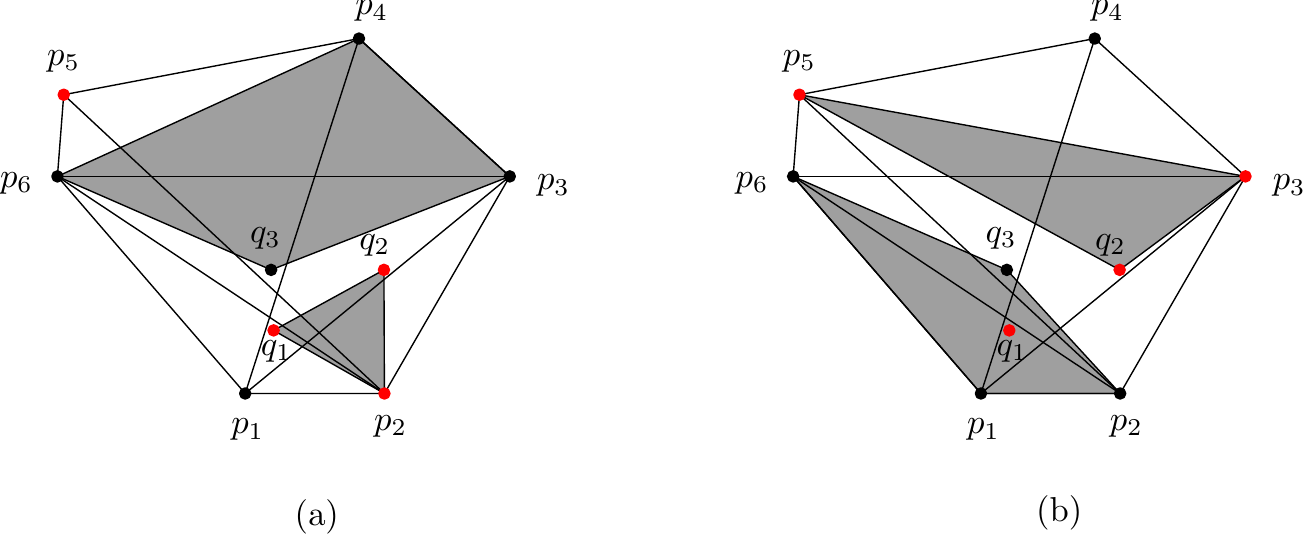}}
\caption{(a) Case 2b. (b) Case 2c.}
\label{case2bc}
\end{figure}

{\bf Case 2}. Set $A$ contains exactly two points of $P$, i.e. $A=\{p_i,p_j,q_c,q_d\}$ 
for some $1\le i<j\le 6$ and $1\le c<d\le 3$. 

{\bf Case 2a}. Suppose $i=1$. By our assumption $j\notin\{2,6\}$. 
We can assume that $j\ne 3$ using $X=\{p_4,p_5,p_6\}$ and Lemma \ref{separ}.
So, $j\in\{4,5\}$.

We also can assume that $q_1\in B$ using $X=\{p_1,q_1\}$ and Lemma \ref{separ}.
Therefore $\{q_2,q_3\}\subset A$.
If $j=4$ (so $A=\{p_1,p_4,q_2,q_3\}$) then take $C=\{p_1,p_3\}$, see Fig.\ref{case2a}(a). 
If $j=5$ (so $A=\{p_1,p_5,q_2,q_3\}$) then take $C=\{p_6,q_2\}$, see Fig.\ref{case2a}(b). 
\begin{figure}[htb]
\centering
\includegraphics{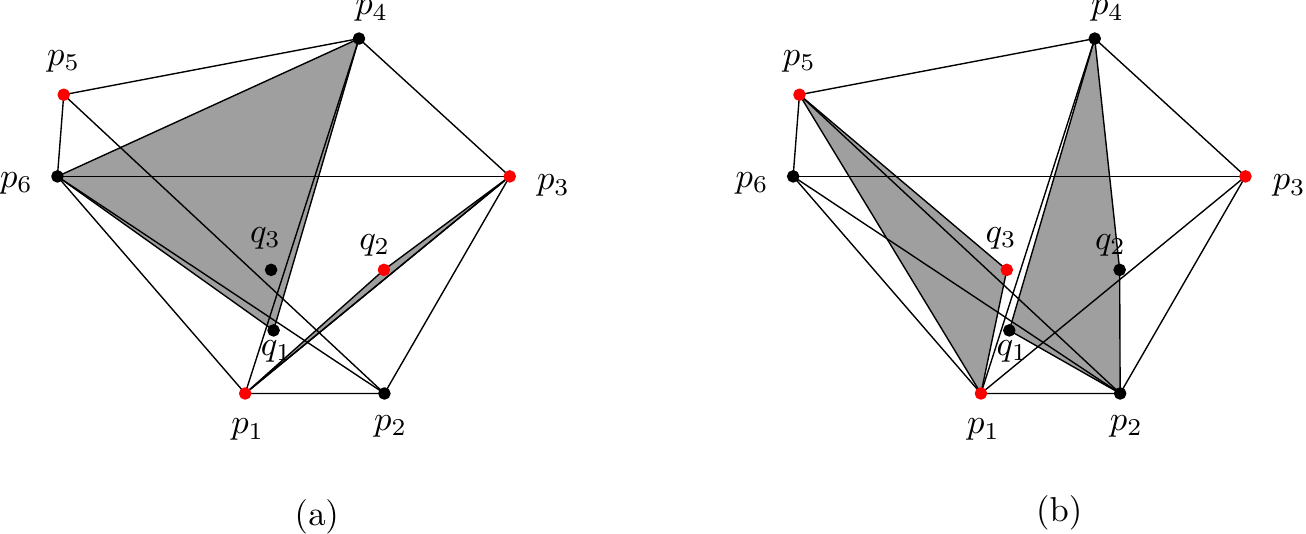}
\caption{Case 3a.}
\label{case3}
\end{figure}

{\bf Case 2b}. Suppose $i=2$. Then $j\ge 4$. 
We can assume $j\ne 4$ using $X=\{p_1,p_5,p_6\}$ and Lemma \ref{separ}.
We can assume $j\ne 6$ using $X=\{p_3,p_4,p_5\}$ and Lemma \ref{separ}.
Thus, $j=5$.
We can assume $q_1\in A$ using $X=\{p_1,p_6,q_1\}$ and Lemma \ref{separ}.

If $q_3\in A$ then $A=\{p_2,p_5,q_1,q_3\}$ we can use $X=\{p_3,p_4,q_2\}$ and Lemma \ref{separ}. 
If $q_2\in A$ then $A=\{p_2,p_5,q_1,q_2\}$ and we can use $C=\{p_1,p_5\}$, see Fig.\ref{case2bc}(a). 

{\bf Case 2c}. Suppose $i,j\in\{3,4,5,6\}$. 
We can assume $q_1,q_2\in A$ using $X=\{p_1,p_2,x\}$ for $x\in\{q_1,q_2\}$ and Lemma \ref{separ}.
Since $p_i$ and $p_{i+1}$ cannot be in $A$ together, there are three options.
If $p_3,p_5\in A$, take
$C=\{p_4,q_1\}$, see Fig.\ref{case2bc}(b).
If $p_4,p_6\in A$ then we can  use $X=\{p_1,p_2,p_3\}$ and Lemma \ref{separ}.
If $p_3,p_6\in A$ then we can  use $X=\{p_3,q_2\}$ and Lemma \ref{separ}. 

{\bf Case 3}. Set $A$ contains exactly three points of $P$, i.e. 
$\{p_1,p_3,p_5\}\subset A$ or $\{p_2,p_4,p_6\}\subset A$.

{\bf Case 3a}.
Suppose $\{p_1,p_3,p_5\}\subset A$.
If $q_1\in A$ then we can  use $X=\{p_1,q_1\}$ and Lemma \ref{separ}. 
If $q_2\in A$, take $C=\{p_2,p_5\}$, see Fig.\ref{case3}(a).
If $q_3\in A$, take $C=\{p_3,p_6\}$, see Fig.\ref{case3}(b).

{\bf Case 3b}.
Suppose $\{p_2,p_4,p_6\}\subset A$.
If $q_1\in A$, take $C=\{p_1,p_4\}$, see Fig.\ref{case3b}(a).
If $q_2\in A$, take $C=\{p_3,p_6\}$, see Fig.\ref{case3b}(b).
If $q_3\in A$, take $C=\{p_2,p_5\}$, see Fig.\ref{case3b}(c).
The lemma follows.
\end{proof}

\begin{figure}[htb]
\centering
\includegraphics[]{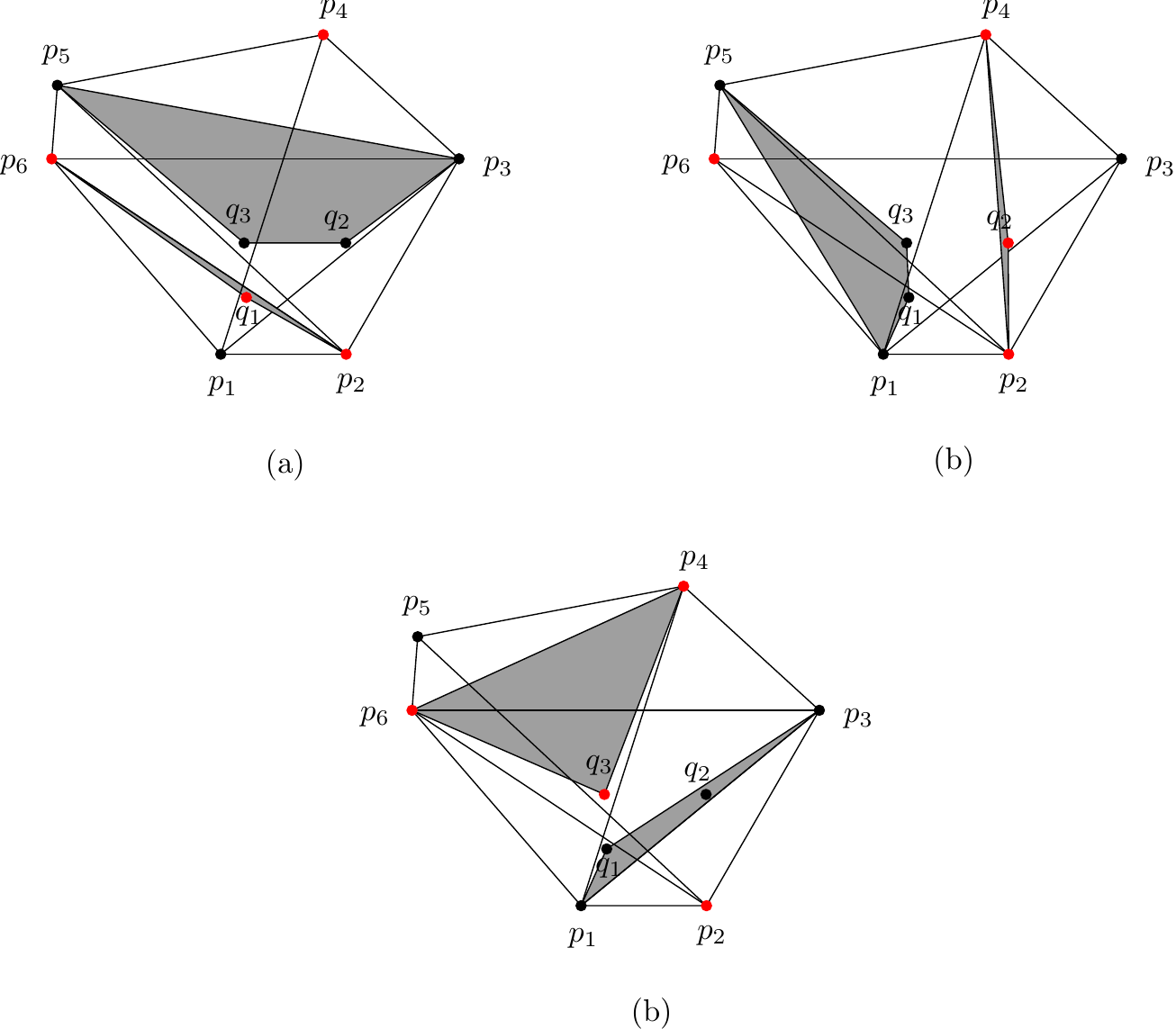}
\caption{Case 3b.}
\label{case3b}
\end{figure}

\begin{cor} \label{lbN222}
$N(2, 2, 2) \ge 10$.
\end{cor}

Sober{\'{o}}n and Strausz~\cite{soberon12} proved an upper bound
$N(d,k,t) \leq (k-1)(t+1)(d+1)+1$ which implies
$N(2, 2, 2) \le 10$. 
Therefore $N(2, 2, 2) = 10$. 

\section{Further Discussion and Open Problems}
\label{diss}

Improving the bounds for Tverberg partitions with tolerance is an interesting problem.  
There are many triples $(d,k,t)$ with a gap between known lower and upper bounds.
Finding sharp bounds for $N(d,k,t)$ is a challenging problem. 
For example, Larman~\cite{larman72} conjectured that $N(d,2,1)=2d+3$ for any $d\ge 1$.
The conjecture is open for $d\ge 5$. 
For dimension $d=5$, the best known upper bound is
$N(5,2,1)\le 13$ by Larman~\cite{larman72} and the best known lower
bound is  $N(5,2,1)\ge 12$ by Garc{\'i}a-Col{\'i}n and Larman~\cite{garcia15}. 

Another possibility for a new sharp bound is for $d=2, k=2$ and tolerance $t=3$. 
The lower bound is $N(2,2,3)\ge 12$ by Theorem \ref{N2kt2}.
The best upper bound is $N(2,2,3)\le 13$ by Sober{\'{o}}n and
Strausz~\cite{soberon12}.
Based on our computer experiments, we conjecture that $N(2,2,3)= 12$. 
One way to improve the upper bound for $d=2, k=2$
and $t=3$ is to study the following core lemma in the
proof of the upper bound for $N(d,k,t)$~\cite{soberon12}.

\begin{lemma} \label{lem1}
Let $p\ge 1$ and $r\ge 0$ be integers, $S\subset\R^p$ be a set of 
$n=p(r+1)+1$ points $a_1,a_2,\dots,a_n$ and $G$ be a group with $|G|\le p$. 
If there is an action of $G$ in a set $S'\subset \R^p$ 
which is compatible with $S\subset S'$, then for each $a_i$ there is a $g_i\in G$ such that the set
$\{g_1 a_1 ,g_2 a_2,\dots,g_n a_n\}$ captures the origin with tolerance $r$.
\end{lemma}

For $d=2, k=2$ and tolerance 3, Lemma \ref{lem1} uses $n=13$ points since $p=(k-1)(d+1)=3$. 
Is this bound for $n$ sharp?

\subsection{Order types and tolerant Tverberg partitions}

Our proof of Theorem \ref{t222} uses the set of nine points shown in Figure \ref{nine}. 
It was computed by a program that we developed for testing the tolerance of a given point set. 
We have applied this program to 158,817 order types for $n=9$ which are provided by 
Aichholzer \etal~\cite{aak-eot-02}. 
Our program found that 155,115 point sets admit a 2-tolerant partition into two sets. 
Only 3,702 point sets do not admit a 2-tolerant partition into two sets and  therefore, 
can be used as examples for proving the lower bound $N(2,2,2)\ge 10$. 

We analyzed these 3,702 order types and found that their convex hulls have sizes between three and six. 
We count the number of different order types for each size of the convex hull, see Table \ref{table}. 
Figures \ref{ch5}, \ref{ch4}, and \ref{ch3} show some order types 
with 5, 4, and 3 points on the convex hull, respectively.
It follows from our computer experiments that any set of nine points in the plane 
with at least seven points on the convex hull admits a 2-tolerant partition into two sets. 
It is interesting to investigate how the size of convex hull of a point set affects lower bounds of $N(d,k,t)$. 

\begin{table}[htb]
\begin{center}
\begin{tabular}{|c|ccccccc|}
\hline
$h$  & $h=3$ & $h=4$ & $h=5$ & $h=6$ & $h=7$ & $h=8$ & $h=9$  \\ 
\hline
\hline
$N_h$  & 1303  & 1554 & 769 & 76 & 0 & 0 & 0  \\
\hline 
\end{tabular}
\end{center}
\caption{3,702 order types of nine points that do not admit a 2-tolerant partition into two sets. 
The point sets are partitioned using the size of the convex hull $h$. 
The size of each group is $N_h$.}
\label{table}
\end{table}

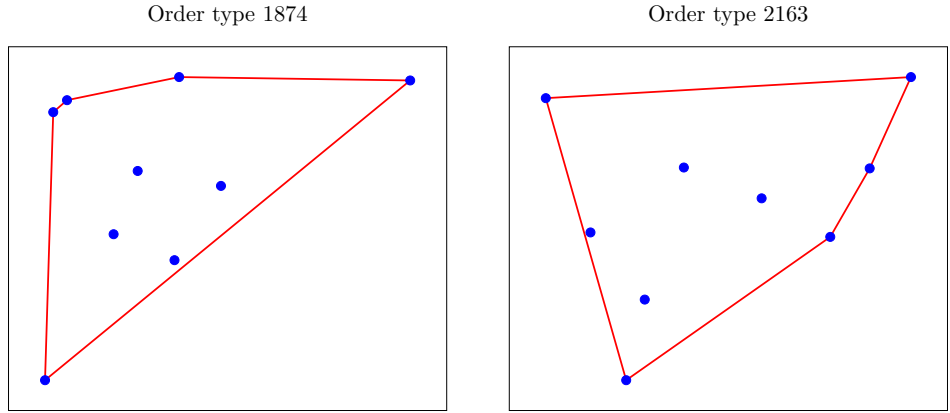
\begin{figure}[htb]
\begin{center}
\begin{tikzpicture}[scale=0.85]
 \begin{axis}[
  title={Order type 1874},
  xtick=\empty, 
  ytick=\empty, 
  scatter/classes={ 
   a={mark=square*,blue}, 
   b={mark=triangle*,red}, 
   c={mark=*,blue}
  }
 ]
  \addplot[scatter,only marks,
   scatter src=explicit symbolic]
   table[meta=label] { 
   x y label 
    9840 6320  c
    11088 53091  c
    13184 55184  c
    20272 31792  c
    23936 42832  c
    29536 27264  c
    30240 59216  c
    36608 40224  c
    65392 58624  c
   };
  \draw[red, thick] (axis cs: 9840, 6320) -- (axis cs: 11088, 53091);
  \draw[red, thick] (axis cs: 11088, 53091) -- (axis cs: 13184, 55184);
  \draw[red, thick] (axis cs: 13184, 55184) -- (axis cs: 30240, 59216);
  \draw[red, thick] (axis cs: 30240, 59216) -- (axis cs: 65392, 58624);
  \draw[red, thick] (axis cs: 65392, 58624) -- (axis cs: 9840, 6320);
 \end{axis}
\end{tikzpicture}
\qquad
\begin{tikzpicture}[scale=0.85]
 \begin{axis}[
  title={Order type 2163},
  xtick=\empty, 
  ytick=\empty, 
  scatter/classes={ 
   a={mark=square*,blue}, 
   b={mark=triangle*,red}, 
   c={mark=*,blue}
  }
 ]
  \addplot[scatter,only marks,
   scatter src=explicit symbolic]
   table[meta=label] { 
   x y label 
    1685 31127  c
    1974 28655  c
    2205 25935  c
    2325 27419  c
    2578 29850  c
    3081 29283  c
    3525 28573  c
    3780 29834  c
    4047 31515  c
   };
  \draw[red, thick] (axis cs: 1685, 31127) -- (axis cs: 4047, 31515);
  \draw[red, thick] (axis cs: 4047, 31515) -- (axis cs: 3780, 29834);
  \draw[red, thick] (axis cs: 3780, 29834) -- (axis cs: 3525, 28573);
  \draw[red, thick] (axis cs: 3525, 28573) -- (axis cs: 2205, 25935);
  \draw[red, thick] (axis cs: 2205, 25935) -- (axis cs: 1685, 31127);
 \end{axis}
\end{tikzpicture}
\end{center}
\caption{Two point sets of nine points each with five points on the convex hull. The point sets have different order types. Both point sets do not admit a 2-tolerant partition into two sets.}
\label{ch5}
\end{figure}

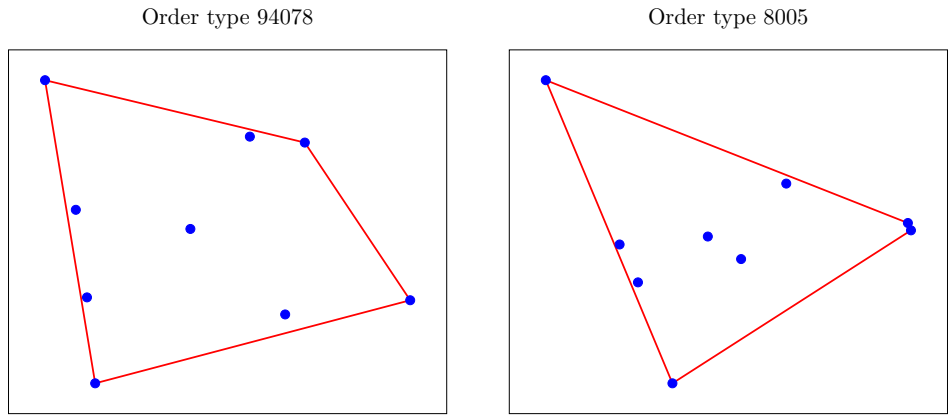
\begin{figure}[htb]
\begin{center}
\begin{tikzpicture}[scale=0.85]
 \begin{axis}[
  title={Order type 94078},
  xtick=\empty, 
  ytick=\empty, 
  scatter/classes={ 
   a={mark=square*,blue}, 
   b={mark=triangle*,red}, 
   c={mark=*,blue}
  }
 ]
  \addplot[scatter,only marks,
   scatter src=explicit symbolic]
   table[meta=label] { 
   x y label 
    5068 6077  c
    5286 4899  c
    5365 4103  c
    5423 3322  c
    6097 4725  c
    6518 5564  c
    6768 3948  c
    6907 5510  c
    7653 4077  c
   };
  \draw[red, thick] (axis cs: 5068, 6077) -- (axis cs: 6907, 5510);
  \draw[red, thick] (axis cs: 6907, 5510) -- (axis cs: 7653, 4077);
  \draw[red, thick] (axis cs: 7653, 4077) -- (axis cs: 5423, 3322);
  \draw[red, thick] (axis cs: 5423, 3322) -- (axis cs: 5068, 6077);
 \end{axis}
\end{tikzpicture}
\qquad
\begin{tikzpicture}[scale=0.85]
 \begin{axis}[
  title={Order type 8005},
  xtick=\empty, 
  ytick=\empty, 
  scatter/classes={ 
   a={mark=square*,blue}, 
   b={mark=triangle*,red}, 
   c={mark=*,blue}
  }
 ]
  \addplot[scatter,only marks,
   scatter src=explicit symbolic]
   table[meta=label] { 
   x y label 
    1191 62310  c
    13953 30288  c
    17130 22903  c
    23090 3225  c
    29211 31836  c
    34965 27444  c
    42777 42156  c
    63813 34475  c
    64344 33041  c
   };
  \draw[red, thick] (axis cs: 1191, 62310) -- (axis cs: 63813, 34475);
  \draw[red, thick] (axis cs: 63813, 34475) -- (axis cs: 64344, 33041);
  \draw[red, thick] (axis cs: 64344, 33041) -- (axis cs: 23090, 3225);
  \draw[red, thick] (axis cs: 23090, 3225) -- (axis cs: 1191, 62310);
 \end{axis}
\end{tikzpicture}
\end{center}
\caption{Two point sets of nine points each with four points on the convex hull. The point sets have different order types. Both point sets do not admit a 2-tolerant partition into two sets.}
\label{ch4}
\end{figure}

\begin{figure}
\begin{center}
\begin{tikzpicture}[scale=0.85]
 \begin{axis}[
  title={Order type 158483},
  xtick=\empty, 
  ytick=\empty, 
  scatter/classes={ 
   a={mark=square*,blue}, 
   b={mark=triangle*,red}, 
   c={mark=*,blue}
  }
 ]
  \addplot[scatter,only marks,
   scatter src=explicit symbolic]
   table[meta=label] { 
   x y label 
    3486 27997  c
    42864 41720  c
    43136 29192  c
    45456 37032  c
    53600 64168  c
    54176 22952  c
    54288 41032  c
    54352 9448  c
    63760 1368  c
   };
  \draw[red, thick] (axis cs: 3486, 27997) -- (axis cs: 53600, 64168);
  \draw[red, thick] (axis cs: 53600, 64168) -- (axis cs: 63760, 1368);
  \draw[red, thick] (axis cs: 63760, 1368) -- (axis cs: 3486, 27997);
 \end{axis}
\end{tikzpicture}
\qquad
\begin{tikzpicture}[scale=0.85]
 \begin{axis}[
  title={Order type 103224},
  xtick=\empty, 
  ytick=\empty, 
  scatter/classes={ 
   a={mark=square*,blue}, 
   b={mark=triangle*,red}, 
   c={mark=*,blue}
  }
 ]
  \addplot[scatter,only marks,
   scatter src=explicit symbolic]
   table[meta=label] { 
   x y label 
    14754 22968  c
    31582 34456  c
    41842 40140  c
    41882 11292  c
    42498 24480  c
    42930 17380  c
    43302 43540  c
    50782 5772  c
    51320 62470  c
   };
  \draw[red, thick] (axis cs: 14754, 22968) -- (axis cs: 51320, 62470);
  \draw[red, thick] (axis cs: 51320, 62470) -- (axis cs: 50782, 5772);
  \draw[red, thick] (axis cs: 50782, 5772) -- (axis cs: 14754, 22968);
 \end{axis}
\end{tikzpicture}

\end{center}
\caption{Two point sets of nine points each with three points on the convex hull. The point sets have different order types. Both point sets do not admit a 2-tolerant partition into two sets.}
\label{ch3}
\end{figure}
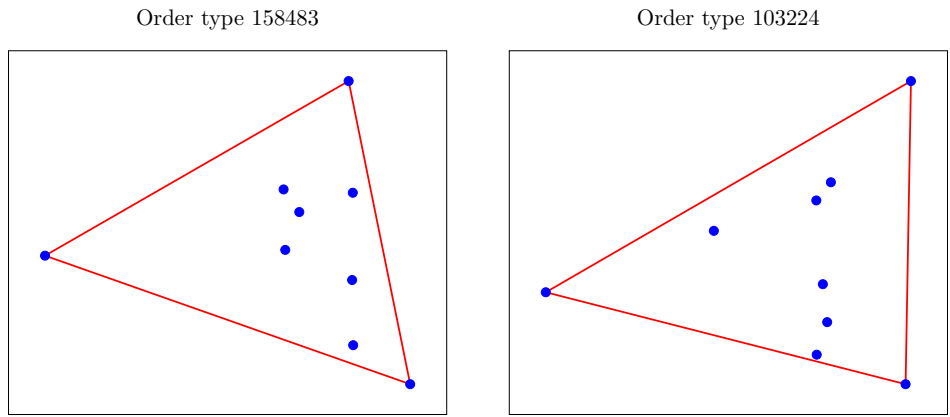

\end{document}